\documentclass{amsart}

\usepackage{amsgen,amstext,amsbsy,amsopn,amsfonts,amssymb}
\usepackage{amsmath}
\usepackage[all]{xy}

\newcommand\eqdef{\stackrel{\mathrm{def}}{=}}

\newcommand{\Mm}{\mathfrak M}
\newcommand\MI{\mathcal{MI}}

\newcommand{\Rr}{\mathbb R}
\newcommand{\Cc}{\mathbb C}

\newcommand{\Zz}{\mathbb Z}

\newcommand{\Id}{\mathbf 1}

\newcommand{\CP}{\overline{\mathbb P}{}}

\newcommand{\R}{\overline{\mathcal R}{}}
\newcommand{\Rreg}{\mathcal R}
\newcommand{\E}{\overline E{}}

\newtheorem{theorem}{Theorem}[section]

\newtheorem{lemma}[theorem]{Lemma}
\newtheorem{proposition}[theorem]{Proposition}

\theoremstyle{definition}

\newtheorem{remark}{Remark}[section]

\DeclareMathOperator\End{end}
\DeclareMathOperator\Tr{tr}
\DeclareMathOperator\im{im}

\begin{document}

\title{Rank-stable limit of completed moduli spaces of instantons}

\author{Jo\~ao Paulo Santos}

\begin{abstract}
In \cite{Nak94}, Nakajima introduced
a resolution of singularities 
of the Donaldson-Uhlenbeck completion of the 
moduli space 
of based instantons over $S^4$.
For $k\leq 4$, we extend this result to $\CP^2$ and compute,
in the rank-stable limit, the homotopy type of these spaces by showing 
that, in this limit, these constructions yield universal bundles.
\end{abstract}

\maketitle
\bibliographystyle{amsplain}

\section{Introduction}

Let $X$ be either the 4-sphere $S^4$ or the complex projective plane $\CP^2$ with reversed orientation.
An instanton on a principal  $SU(r)$-bundle $P$ over $X$ 
is a
connection $\nabla$ on $P$ whose curvature $F_\nabla$ is 
anti-self-dual with respect to the Hodge $*$-operator.
The moduli space
of instantons based at a point $x_0\in X$ is the quotient of
the space of instantons by the action of the gauge group of
automorphisms of $P$ which are the identity on the fiber over $x_0$. This moduli space
depends only on $r$ and on the second Chern class $c_2(P)=k$, so we represent it by $\MI_k^r(X)$.
For $r'>r$, there is
a map $\MI_k^r(X)\to\MI_k^{r'}(X)$ induced by the inclusion $\Cc^r\to\Cc^{r'}$,
and in \cite{Kir94}, \cite{San95}, \cite{BrSa97} it was shown that the direct limit
$\MI_k^\infty(X)$ has the homotopy type of $BU(k)$, for $X=S^4$, and
the homotopy type of $BU(k)\times BU(k)$, for $X=\CP^2$.

In \cite{Nak94}, Nakajima introduced a space
$\overline\Mm{}_{k,\zeta}^r(S^4)$ with a dense subspace
$\Mm_{k,\zeta}^r(S^4)$ isomorphic to $\MI_k^\infty(S^4)$, and showed that
$\overline\Mm{}_{k,\zeta}^r(S^4)$ is a resolution of singularities
of the usual Donaldson-Uhlenbeck completion of $\MI_k^r(S^4)$.
In this paper we extend, for $k\leq 4$, Nakajima's construction to $\CP^2$ (Theorem~\ref{thm32}).
For dimensional reasons, for $X=S^4,\CP^2$, the inclusions
$\Mm_k^\infty(X)\to\overline\Mm{}_k^\infty(X)$ are homotopy equivalences; 
generalizing the approach in
\cite{San95}, \cite{BrSa97}, 
we obtain universal bundles over $\overline\Mm{}_k^\infty(X)$
which restrict to universal bundles over $\Mm_k^\infty(X)$
(Theorems~\ref{thm21} and \ref{thm31}).

\section{Instantons on $S^4$}

We begin by sketching the \textsc{adhm} description of instantons on $S^4$ introduced in \cite{ADHM78},
and the space $\overline\Mm{}_{k,\zeta}^r(S^4)$ introduced in \cite{Nak94}.
Let $W$ be a $k$-dimensional hermitian vector space.
Let $\R_k^r(S^4)$ 
be the space of configurations $(a_1,a_2,b,c)$ where
$a_i\in\End(W)$ (with $i=1,2$), $b\in\hom(\Cc^r,W)$ and $c\in\hom(W,\Cc^r)$,
obeying the integrability condition
$[a_1,a_2]+bc=0$.
Let $\Rreg_k^r(S^4)\subset\R_k^r(S^4)$ 
be the subspace of 4-tuples
obeying the two non-degeneracy conditions:
\begin{description}
\item[C1] There is no proper subspace $W'\subset W$ such that
$\im b\subset W'$ and $a_i(W')\subset W'$ (for $i=1,2$);
\item[C2] There is no nonempty subspace $W'\subset W$ such that
$W'\subset\ker c$ and $a_i(W')\subset W'$ (for $i=1,2$).
\end{description}
The unitary group $U(W)$ acts on $\R_k^r(S^4)$ by
$g\cdot(a_1,a_2,b,c)=(ga_1g^{-1},ga_2g^{-1},
gb,cg^{-1})$, with moment map equation:
\begin{equation}\label{eq-mmS4}
\mu(a_1,a_2,b,c)=[a_1,a_1^*]+[a_2,a_2^*]+bb^*-c^*c=0\,.
\end{equation}
Let $\zeta$ be a non-zero real parameter.
It was proven in \cite{Nak94} that
$U(W)$ acts freely on $\mu^{-1}(-\zeta)$,
the quotient $\Mm_{k,\zeta}^r(S^4)\eqdef\mu^{-1}(-\zeta)\cap\Rreg_k^r(S^4)/U(W)$ is isomorphic to
the moduli space $\MI_k^r(S^4)$ and the quotient
$\overline\Mm{}_{k,\zeta}^r(S^4)\eqdef\mu^{-1}(-\zeta)/U(W)$ 
is a resolution of singularities
of the Donaldson-Uhlenbeck completion of $\MI_k^r(S^4)$.

Now we consider the limit when $r\to\infty$. We fix $\zeta\in\Rr$ and we let
$\E_k^r=\mu^{-1}(-\zeta)\subset\R_k^r$ and $E_k^r=\E_k^r\cap\Rreg_k^r(S^4)$.
For $r'>r$, the inclusion $\Cc^r\to\Cc^{r'}$ induces a map
$i_{r,r'}\colon \E_k^r\to \E_k^{r'}$.
This map preserves the subspace $E_k^r$,
and descends to the quotient to give a map of pairs
$\bigl(\overline\Mm{}_k^r,\Mm_k^r\bigr)\to\bigl(\overline\Mm{}_k^{r'},\Mm_k^{r'}\bigr)$.
We define the rank-stable moduli space by taking the direct limit over $r$:
$\Mm_k^\infty(S^4)=\smash[b]{\lim\limits_{\substack{\longrightarrow\\r}}\Mm_k^r(S^4)}$;
we define in the same way $\overline\Mm{}_k^\infty$, $\E_k^\infty$
and $E_k^\infty$.

\begin{theorem}\label{thm21}
The maps $\E_k^\infty\to \overline\Mm{}_k^\infty$ and 
$E_k^\infty\to\Mm_k^\infty$ are universal $U(k)$ bundles and
the inclusion map $j\colon \Mm_k^\infty(S^4)\to\overline\Mm{}_k^\infty(S^4)$ is a homotopy equivalence.
\end{theorem}
\begin{proof}
We first show that the space $\E_k^\infty$ 
is contractible.
It is enough to show that its homotopy groups are trivial,
and this will follow
if we show that the inclusions 
$i_{r,r+2k}\colon \E_k^r\to \E_k^{r+2k}$
are null-homotopic. 
We first identify $W$ with $\Cc^k$
via a complex hermitian isomorphism; 
then we pick $\zeta_b,\zeta_c\in\Rr^+$ such that $\zeta_c-\zeta_b=\zeta$ and
define a homotopy $h\colon \E_k^r\times[0,1]\to\E_k^{r+2k}$ by
\[
h_t(a_1,a_2,b,c)=\left(\sqrt{1-t}\,a_1,\sqrt{1-t}\,a_2,b_t,c_t\right)
\]
where, using matrix notation,
\[
b_t=\Bigl(\begin{matrix}\sqrt{1-t}\,b&0&\sqrt{t\zeta_b}\,\Id\end{matrix}\Bigr)
\qquad\text{and}\qquad
c_t=\Bigl(\begin{matrix}\sqrt{1-t}\,c& \sqrt{t\zeta_c}\,\Id& 0\end{matrix}\Bigr)^T.
\]
It is a direct verification that $h$ is a well defined homotopy
between $i_{r,r+2k}$ and a constant map. Furthermore, since
$b_t$, $c_t$ both have maximal rank for $t\neq0$, 
conditions C1 and C2 are both satisfied so
$h$ restricts to a homotopy
$E_k^r\times[0,1]\to E_k^{r+2k}$.
Hence we also conclude that $E_k^\infty$
is contractible. Since the $U(W)$ action is free, $\E_k^\infty$ and $E_k^\infty$
are universal $U(W)$ bundles. To conclude the proof we apply
the five lemma to the long exact sequences of homotopy groups
associated with the principal bundles $\E_k^\infty$ and $E_k^\infty$ 
to conclude that $j$ induces isomorphisms on all homotopy groups, hence it is
a homotopy equivalence.
\end{proof}

\section{Instantons on $\CP^2$}

We now look at instantons on $\CP^2$.
We begin by sketching the monad description
of the moduli space introduced
in \cite{Kin89}.
Let $W_0$, $W_1$ be $k$-dimensional hermitian vector spaces.
Let $\R_k^r(\CP^2)$ be the space of configurations
$(a_1,a_2,d,b,c)$ where
$a_i\in\hom(W_1,W_0)$ (with $i=1,2$), $d\in\hom(W_0,W_1)$,
$b\in\hom(\Cc^r,W_0)$ and $c\in\hom(W_1,\Cc^r)$,
subject to the condition $a_1(W_1)+a_2(W_1)+b(\Cc^r)=W_0$ and
obeying the integrability equation
\begin{equation}\label{eqintegrability}
a_1da_2-a_2da_1+bc=0\,.
\end{equation}
Let $\Rreg_k^r(\CP^2)\subset\R_k^r(\CP^2)$ be the subspace
of configurations obeying the non-degeneracy conditions:
\begin{description}
\item[C1'] There are no proper subspaces $V_0'\subset W_0$ and $V_1'\subset W_1$
such that
$\dim V_0'=\dim V_1'$,
$\im b\subset V_0'$, $d(V_0')\subset V_1'$ and $a_i(V_1')\subset V_0'$ (with $i=1,2$);
\item[C2'] There are no nonempty subspaces $V_0\subset W_0$ and $V_1\subset W_1$
such that
$\dim V_0=\dim V_1$,
$V_1\subset\ker c$, $d(V_0)\subset V_1$ and $a_i(V_1)\subset V_0$ (with $i=1,2$).
\end{description}
\begin{remark}
The condition $\dim V_0'=\dim V_1'$ in~C1' can be replaced
by $\dim V_0'\leq\dim V_1'$: if ${\dim}V_0'<{\dim}V_1'$
we can always replace $V_0'$ with a bigger space to obtain equality. In the 
same way, the condition $\dim V_0=\dim V_1$ in~C2' can be replaced
by $\dim V_0\leq\dim V_1$.
\end{remark}

The group $U(W_0)\times U(W_1)$ acts on $\R_k^r(\CP^2)$ by
\begin{equation}\label{eq-action}
(g_0,g_1)\cdot(a_1,a_2,d,b,c)=(g_0a_1g_1^{-1},g_0a_2g_1^{-1},g_1dg_0^{-1},
g_0b,cg_1^{-1})\,.
\end{equation}
We now consider the moment map $\mu=(\mu_0,\mu_1)$; the moment map
equations are
\begin{equation}\label{eq-mmP2}
\begin{cases}
\mu_0= a_1a_1^*+a_2a_2^*+bb^*-\Id_{W_0}=0\\
\mu_1=
\bigl[\,da_1,(da_1)^*\,\bigr]+\bigl[\,da_2,(da_2)^*\,\bigr]-a_1^*a_1-a_2^*a_2+db(db)^*-c^*c+\Id_{W_1}=0
\end{cases}
\end{equation}
(where $\Id_{W_i}$ denotes the identity map on $W_i$).
It was shown in \cite{Kin89} that 
the restriction of the action \eqref{eq-action} to $\mu^{-1}(0,0)\cap\Rreg_k^r$ 
is free and the quotient, which we represent by $\Mm_k^r(\CP^2)$, is isomorphic to
the moduli space $\MI_k^r(\CP^2)$. Furthermore, the quotient
$\overline\Mm{}_k^r(\CP^2)\eqdef\mu^{-1}(0,0)/U(W_0)\times U(W_1)$ is isomorphic to the
Donaldson-Uhlenbeck completion of $\MI_k^r(\CP^2)$.

\subsection{Resolution of singularities}

We now perturb the second
moment map equation by introducing a real
parameter $\zeta$.  We will show that, for $k\leq 4$, the space
$\overline\Mm{}_{k,\zeta}^r(\CP^2)=\mu^{-1}(0,\zeta)/U(W_0)\times U(W_1)$
is a resolution of singularities of $\overline\Mm{}_k^r(\CP^2)$. First we need some Lemmas:

\begin{lemma}\label{lemmamaxrank}
Let $(a_1,a_2,d,b,c)\in\mu^{-1}(0,\zeta)$ with $|\zeta|<1$.
Then both matrices
$\bigl(\, a_1\ a_2\ b\, \bigr){}^T$ and $\bigl(\, a_1\ a_2\ c\, \bigr)$
have maximal rank.
\end{lemma}
\begin{proof}
The equation $\mu_0=0$ shows that $\ker a_1^*\cap\ker a_2^*\cap\ker b^*=0$, and hence
$\im a_1\oplus\im a_2\oplus\im b=W_0$. From the moment map equations 
we deduce the equation
\begin{equation}\label{eq-m0+m1}
-d\mu_0d^*+\mu_1=
\textstyle-\sum\limits_{i=1}^2a_i^*(d^*d+\Id_{W_0})a_i-c^*c+\Id_{W_1}+dd^*=\zeta\Id_{W_1}\,.
\end{equation}
Since $|\zeta|<1$,
it follows that $\ker a_1\cap\ker a_2\cap\ker c=0$, which concludes the proof.
\end{proof}

\begin{lemma}\label{lemmaC1C2}
Condition \emph{C1'}
always holds if $\zeta>0$
and condition \emph{C2'} always holds if $\zeta<0$.
\end{lemma}
\begin{proof}
If we take orthogonal decompositions $W_j=V_j\oplus V_j'$ (with $j=0,1$)  
and write in block form:
\[
a_i=\begin{pmatrix}a_i^{11}&a_i^{12}\\a_i^{21}&a_i^{22}\end{pmatrix}
\quad
d=\begin{pmatrix}d^{11}&d^{12}\\d^{21}&d^{22}\end{pmatrix}
\quad
b=\begin{pmatrix}b^1\\b^2\end{pmatrix}
\quad
c=\begin{pmatrix}c^1\\c^2\end{pmatrix}^T,
\]
taking the trace of the $(1,1)$-component of the 
second moment map equation we obtain
\begin{multline*}
\textstyle\sum\limits_{i=1}^2\|d^{11}a_i^{12}+d^{12}a_i^{22}\|^2
+\sum\limits_{i=1}^2\|a_i^{12}\|^2
+\|b^1\|^2+\|d^{11}b^1\|^2+\|d^{12}b^2\|^2
\\
\textstyle=\sum\limits_{i=1}^2\|d^{21}a_i^{11}
+d^{22}a_i^{21}\|^2
+\sum\limits_{i=1}^2\|a_i^{21}\|^2
+\|c^1\|^2+\zeta\dim V_0\,.
\end{multline*}
If condition C1' fails with the supspaces $V_0'$, $V_1'$ then 
$a_i^{12}=d^{12}=b^1=0$ and hence $\zeta\leq 0$. Similarly,
if condition C1' fails with the supspaces $V_0$, $V_1$ then $\zeta\geq 0$,
which concludes the proof.
\end{proof}

\begin{lemma}\label{lemmafree}
If $\zeta\notin\Zz$,
the $U(W_0)\times U(W_1)$ action on $\mu^{-1}(0,\zeta)$ is free.
\end{lemma}
\begin{proof}
Suppose $(g_0,g_1)\in U(W_0)\times U(W_1)$ 
stabilizes $(a_1,a_2,d,b,c)\in\mu^{-1}(0,\zeta)$. 
Then $g_0a_i=a_ig_1$ (with $i=1,2$),
$g_1d=dg_0$, $g_0b=b$ and $cg_1=c$. For each $\lambda\in\Cc$
let $W_j(\lambda)\subset W_j$ (with $j=0,1$)
denote the $\lambda$-eigenspace of $g_j$ . Then, the linear maps
$a_i$, $d$, $a_i^*$, $d^*$ preserve the spaces $W_0(\lambda)$,
$W_1(\lambda)$
and, for $\lambda\neq1$,
\[
W_0(\lambda)\subset\ker b^*\,,\quad
W_1(\lambda)\subset\ker(db)^*
\quad\text{and}\quad W_1(\lambda)\subset\ker c\,.
\]
It follows, for $\lambda\neq1$, that,  by
restricting the perturbed moment map equations 
to $W_j(\lambda)$, 
we get
\[
\begin{cases}
a_1a_1^*+a_2a_2^*=\Id_{W_0}\,,\\
\bigl[\,da_1,(da_1)^*\,\bigr]+\bigl[\,da_1,(da_1)^*\,\bigr]-a_1^*a_1-a_2^*a_2=(1-\zeta)\Id_{W_1}\,.
\end{cases}
\]
Taking the trace, since ${\Tr}(a_1a_1^*+a_2a_2^*)={\Tr}(a_1^*a_1+a_2^*a_2)$
we get ${\dim}W_0(\lambda)=(1-\zeta){\dim}W_1(\lambda)$. 
Since $\zeta\notin\Zz$, this is impossible unless
$W_0(\lambda)=W_1(\lambda)=0$ for any $\lambda\neq1$ which implies that
$(g_0,g_1)=(\Id_{W_0},\Id_{W_1})$. We conclude that the action is free.
\end{proof}

\begin{proposition}
For $k\leq 4$ and $0<|\zeta|<1$,
the space $\overline\Mm{}_{k,\zeta}^r(\CP^2)$ 
is a smooth manifold of dimension $4kr$.
\end{proposition}
\begin{proof}
From equation~\eqref{eqintegrability} and
Lemma~\ref{lemmafree}, it is enough
to show that the map $f(a,d,b,c)=a_1da_2-a_2da_1+bc$ 
has surjective
diferencial on $\mu^{-1}(0,\zeta)$. If $\mathrm df$ were not surjective,
we could find $x\in\End(W_0,W_1)$ such that
\begin{equation}\label{eq-df}
cx=0,\quad xb=0,\quad a_1xa_2=a_2xa_1,\quad xa_id=da_ix\quad  (i=1,2).
\end{equation}
Then, it follows from Lemma \ref{lemmamaxrank}
that $\ker(a_1x)\cap\ker(a_2x)=\ker x$ and $\im(xa_1)\oplus\im(xa_2)=\im x$.
We will show that, in this case, both conditions C1' and C2' fail,
which contradicts Lemma \ref{lemmaC1C2}: we will show that
the pair of subspaces
\[
\textstyle V_0'=\ker x=\bigcap\limits_{i=1,2}\ker(a_ix)\subset W_0\qquad\text{and}\qquad
V_1'=\bigcap\limits_{i=1,2}\ker(xa_i)\subset W_1
\]
is a counterexample for C1' and the pair of subspaces
\[
\textstyle V_0=\bigoplus\limits_{i=1,2}{\im}(a_ix)\subset W_0\qquad\text{and}\qquad
V_1={\im}x=\bigoplus\limits_{i=1,2}{\im}(xa_i)\subset W_1
\]
is a counterexample for C2'. 
We just need to check that $\dim V_0'\leq\dim V_1'$ and $\dim V_0\leq\dim V_1$, since the other conditions are straightforward. We start with 
the inequality $\dim V_0'=\dim\ker x\leq\bigcap_i\ker(xa_i)=\dim V_1'$.
We will need the following result:
\begin{lemma}\label{lemma41}
Let $v_i\in\ker(xa_i)$, with $i=1,2$, 
and assume that $xa_2v_1=xa_1v_2$. Then
$xa_2v_1=xa_1v_2=0$.
\end{lemma}
\begin{proof}
From equation \eqref{eq-df}, we get
$a_2xa_2v_1=a_2xa_1v_2=a_1xa_2v_2=0$ and $a_1xa_2v_1=a_2xa_1v_1=0$,
hence $a_2v_1\in\bigcap_i\ker(a_ix)=\ker x$ and so
\renewcommand\qed{}
$xa_2v_1=0$, which concludes the proof of the Lemma.
\end{proof}

Equation \eqref{eq-df} implies that $xa_1$ and $xa_2$ commute,
and so do $a_1x$ and $a_2x$. Let $k=\dim W_0=\dim W_1$. Given $\lambda,\mu\in\Cc$,
consider the generalized eigenspaces
\[
V_0(\lambda,\mu)=\ker(a_1x-\lambda)^k\cap\ker(a_2x-\mu)^k\quad\text{and}\quad
V_1(\lambda,\mu)=\ker(xa_1-\lambda)^k\cap\ker(xa_2-\mu)^k.
\]
For any $(\lambda,\mu)\neq(0,0)$, the map $x$ induces an isomorphism $V_0(\lambda,\mu)\stackrel{\simeq}{\to}V_1(\lambda,\mu)$
and hence we also have $\dim V_0(0,0)=\dim V_1(0,0)$. So, we may assume without loss of generality that
both $xa_1$ and $xa_2$ are nilpotent, otherwise we just restrict to the spaces $V_0(0,0)$ and $V_1(0,0)$.
To show that $\dim\ker x\leq\dim\bigcap_i\ker(xa_i)$
we consider four cases:
\begin{description}
\item[Case 1] We first assume that $\ker(xa_1)\subset\ker(xa_2)$
(the case $\ker(xa_2)\subset\ker(xa_1)$ is analogous).
Since the map $a_1\colon W_1/\ker(xa_1)\to W_0/\ker x$ is injective
and $\dim W_0=\dim W_1$ we imediately get
$\dim\ker x\leq\dim \ker(xa_1)=\dim\bigcap_i\ker(xa_i)$ which concludes 
the proof.
\item[Case 2] We now show that the equality
$\dim\bigcap\ker(xa_i)=1$ can only
occur if we are in Case 1. Because, otherwise, 
pick a generator $v_0$ of $\bigcap\ker(xa_i)$; by restricting
$xa_1$ to $\ker(xa_2)$, since $\dim\ker(xa_2)>1$ we can find some vector $v_2\in\ker(xa_2)$
with $xa_1v_2=v_0$, and similarly there is a $v_1\in\ker(xa_1)$
with $xa_2v_1=v_0$, contradicting Lemma \ref{lemma41}.
\item[Case 3] Now we consider the case where $k\leq 3$. Then,
by a dimensional argument we see that we are either in Case~1 or in
Case~2.
\item[Case 4] Finally we let $k=4$ and we assume that we are neither in 
Case~1 nor in Case~2. Then necessarily
$\dim\bigcap\ker(xa_i)=2$ and $\dim\ker(xa_i)=3$ with $i=1,2$.
Pick vectors $v_1\in\ker(xa_1)-\ker(xa_2)$ and
$v_2\in\ker(xa_2)-\ker(xa_1)$. 
The vectors $xa_1v_2$ and $xa_2v_1$ are both non-zero hence, by 
Lemma \ref{lemma41},
$xa_1v_2$ and $xa_2v_1$ are linearly independent. It follows that
$a_1v_2$ and $a_2v_1$ are linearly independent in 
$W_0/\ker x$ and hence $\dim\ker x\leq 2$ which concludes the proof
for condition C1'.
\end{description}
For condition C2', we only need to show that $\dim V_1^*\leq \dim V_0^*$.
The proof is completely analogous to the proof for condition C1'.
\end{proof} 

Equation \eqref{eq-action} also defines an action of
the group $GL(W_0)\times GL(W_1)$ on $\R_k^r$, and it was shown in \cite{Kin89}
that the algebraic-geometric quotient $\R_k^r/\!/GL(W_0)\times GL(W_1)$ 
is isomorphic to $\overline\Mm{}_{k}^r(\CP^2)$.
Combining this isomorphism with the inclusion $\mu^{-1}(0,\zeta)\subset\R_k^r$ 
we get a map
$\pi\colon\overline\Mm{}_{k,\zeta}^r(\CP^2)\to\overline\Mm{}_{k}^r(\CP^2)$.

\begin{theorem}\label{thm32}
For $k\leq 4$ and $0<|\zeta|<1$, 
the map $\pi$ defined above is a resolution of singularities:
\begin{enumerate}
\item $\pi$ is proper;
\item $\pi$ induces an isomorphism 
$\pi^{-1}(\Mm_{k}^r)\stackrel{\simeq}{\to}\Mm_{k}^r$;
\item $\pi^{-1}(\Mm_{k}^r)$ is dense in $\overline\Mm{}_{k,\zeta}^r$.
\end{enumerate}
\end{theorem}
\begin{proof}\hfill
\begin{enumerate}
\item We assume, by contradiction, that there is a sequence $(x_k)$ in $\overline\Mm{}_{k,\zeta}^r(\CP^2)$ without any
convergent subsequence and such that $\pi(x_k)$ converges.
Let $m_k=(a_{1k},a_{2k},d_k,b_k,c_k)\in\mu^{-1}(0,\zeta)$ 
be a representative of $x_k$;
then $\|m_k\|\to+\infty$. 
Let $p:\R_k^r(\CP^2)\to\R_k^r(S^4)$ be given by
$p(a_1,a_2,d,b,c)=(da_1,da_2,db,c)$. 
Taking the trace of the moment map equations  \eqref{eq-mmP2} and \eqref{eq-m0+m1} we see that
$\|a_{ik}\|$, $\|b_k\|$ and $\|da_{ik}\|^2+\|c_k\|^2-\|d_k\|^2$ are bounded so 
$\|p(m_k)\|\to+\infty$. Passing to a subsequence if necessary, we may assume that $p(m_k)/\|p(m_k)\|$ converges;
let $m_\infty$ be the limit; since $\mu(m_k)=(0,\zeta)$,
dividing by $\|m_k\|^2$ and taking the limit we see that
$m_\infty$ obeys the moment map equation \eqref{eq-mmS4}. From general properties of the moment map, the $GL$-orbit of $m_\infty$
is closed. Now let $n_k\in\mu^{-1}(0,0)$ be a representative of $\pi(x_k)$;
since $m_k$ and $n_k$ represent the same element in $\R_k^r(\CP^2)/\!/GL\times GL$,
$p(m_k)/\|m_k\|$ and $p(n_k)/\|m_k\|$
represent the same element in $\R_k^r(S^4)/\!/GL$. Since
$\pi(x_k)$ converges, $p(n_k)/\|p(m_k)\|\to 0$; but then, $0$ is in the closure of the $GL$-orbit
of $m_\infty$ which implies that $m_\infty=0$, which is a contradiction since $\|m_\infty\|=1$.
\item The proof is the same as the one in \cite{NakALE}, \cite{Kro89}.
\item We just need to check that no component of
$\overline\Mm{}_{k,\zeta}^r$ is mapped into the singular set. 
This follows from dimensional
reasons: every component has dimension $4kr$ and
$\pi^{-1}(\overline\Mm_k^r-\Mm_k^r)$ has dimension at most $4kr-2r$.\qedhere
\end{enumerate}
\end{proof}

\subsection{The rank-stable limit}

We now fix $\zeta\in\Rr$ with $|\zeta|<1$ and let $\E_k^r=\mu^{-1}(0,\zeta)$ and
$E_k^r=\E_k^r\cap\Rreg_k^r$.
Similarly to the case $X=S^4$,
we have maps of pairs 
$\bigl(\overline\Mm{}_k^r,\Mm_k^r\bigr)\to\bigl(\overline\Mm{}_k^{r'},\Mm_k^{r'}\bigr)$
and we define $\Mm_k^\infty$, $\overline\Mm{}_k^\infty$, $E_k^\infty$ and $\E_k^\infty$ as the direct limits 
when $r\to\infty$.

\begin{theorem}\label{thm31}
The maps $\E_k^\infty\to \overline\Mm{}_k^\infty$ and 
$E_k^\infty\to\Mm_k^\infty$ are universal $U(W_0)\times U(W_1)$ bundles and
the inclusion map $j\colon \Mm_k^\infty(\CP^2)\to\overline\Mm{}_k^\infty(\CP^2)$ is a homotopy equivalence.
\end{theorem}
\begin{proof}
The proof follows the same lines as the proof of Theorem \ref{thm21}.
We first show that the inclusions $i_{r,r+3k}\colon \E_k^r\to\E_k^{r+3k}$
are null-homotopic. 
First we identify $W_0$ and $W_1$ with $\Cc^k$
via a complex hermitian isomorphism; 
then 
we define a homotopy $h\colon \E_k^r\times[0,1]\to\E_k^{r+3k}$ by
\[
h_t(a_1,a_2,d,b,c)=\left(\,\sqrt{1-t}\,a_1\,,\,\sqrt{1-t}\,a_2\,,\,d\,,\,b_t\,,\,c_t\,\right)
\]
where
\[
b_t=\Bigl(\begin{matrix}\sqrt{1-t}\,b&0&\sqrt{t}\,\Id&0\end{matrix}\Bigr)\quad
\text{and}\quad
c_t=\smash[t]{\Bigl(\begin{matrix}\sqrt{1-t}\,c & \sqrt{t}\,d^* & 0 & \sqrt{t\zeta}\,\Id\end{matrix}\Bigr)^T}\,.
\]
It is a direct verification that $h$ is a well defined homotopy
between $i_{r,r+3k}$ and the map $f\colon \E_k^r\to E_k^{r+3k}$
given by $f(a_1,a_2,d,b,c)=(0,0,d,b_1,c_1)$.
Furthermore, $h$ restricts to a homotopy
$E_k^r\to E_k^{r+3k}$.
We now follow $h$ with another homotopy 
$\tilde h\colon \E_k^{r}\times[0,1]\to\E_k^{r+3k}$ defined by
\[
\tilde h_t(a_1,a_2,d,b,c)=\left(0,0,(1-t)\,d,b_1,\tilde c_t\right),
\]
where
\[
\tilde c_t=\smash[t]{\Bigl(\begin{matrix}0&(1-t)\,d^*&0&\sqrt\zeta \,\Id\end{matrix}\Bigr)^T}.
\]
The map $\tilde h$ is then a homotopy between the map $f$ and a constant map. 
We conclude that the maps $i_{r,r+3k}$ are null-homotopic, from which it follows that the
space $\E_k^\infty$ is contractible. By restricting $h$ and $\tilde h$ to 
$E_k^r\times[0,1]$, we also conclude that $E_k^\infty$ is contractible.
As in the proof of Theorem~\ref{thm21}, we conclude that 
the map $j$ is a homotopy equivalence.
\end{proof}

\bibliography{bi.bib}

\end{document}